\documentclass{amsart}

\newtheorem{theorem}{Theorem}

\newtheorem{lemma}{Lemma}

\theoremstyle{definition}
\newtheorem{definition}{Definition}

\theoremstyle{remark}
\newtheorem{remark}{\bf {Remark}}

\numberwithin{equation}{section}

\usepackage{color}
\usepackage[top=2.5cm, bottom=2cm, left=2.3cm, right=2.5cm]{geometry}

\usepackage{esvect}
\usepackage{relsize}

\usepackage{ amssymb }
\usepackage{ marvosym }

\newcommand{\Nset}{\mathbb{N}}

\newcommand{\Rset}{\mathbb{R}}

\title{On the Baire space $D^{\kappa}$ having $\omega_1$-strongly compact weight.}
\author{Ana S. Mero\~no }

\begin{document}
\maketitle 

\bigskip

{\it Abstract}. We prove that on the Baire space $(D^{\kappa},\pi)$, $\kappa \geq \omega_0$ where $D$ is a uniformly discrete space having $\omega _1$-strongly compact cardinal and $\pi$ denotes the product uniformity on $D^\kappa$, there exists a $z_u$-filter $\mathcal{F}$ being Cauchy for the uniformity $e\pi$ having as a base all the countable uniform partitions of $(D^\kappa,\pi)$, and failing the countable intersection property.  This fact is equivalent to the existence of a non-vanishing real-valued  uniformly continuous function $f$ on $D^{\kappa}$ for which the inverse function $g=1/f$ cannot be continuously extended to the completion of  $(D^{\kappa _0},e\pi)$. This does not happen when the cardinal of $D$ is strictly smaller than the first Ulam-measurable cardinal.
 
\bigskip

\Large

\begin{section}{Introduction}

Given a uniform space $(X,\mu)$, consider the completion of $X$ endowed with the weak uniformity $wU_\mu(X)$ induced by all the real-valued uniformly continuous functions on $(X,\mu)$ (see \cite{willard}). The topological space obtained in this completion is a realcompactification of $X$. More precisely, it is the smallest realcompactification of $(X,\mu)$, in the usual order of realcompactifications (\cite{engelking.realcompact}), such that every real-valued uniformly continuous function $f\in U_\mu (X)$ can be continuously (and uniquely) extended to it. 

We denote this realcompactification by $H(U_\mu (X))$, following  \cite{merono.Samuel.uniform}, where it is called the  {\it Samuel realcompactification} of $(X,\mu)$ since it is defined by means of the family of all the real-valued uniformly continuous functions in parallel to the {\it Samuel compactification} $s_\mu X$ (\cite{samuel}), which is the compactification of $(X,\mu)$ obtained by doing the completion of $(X, wU^{*}(X))$, where  $wU^{*}(X)$ is the weak uniformity induced by all the bounded real-valued uniformly continuous functions on $(X,\mu)$. The Samuel realcompaction has been well-studied in \cite{reynolds}, \cite{merono.Samuel.uniform} and \cite{husek}, where the uniform spaces $(X,\mu)$ being {\it Samuel realcompact}, that is, satisfying that $X=H(U_\mu(X))$, are characterized.

In general, the Samuel realcompactification of a uniform space $(X,\mu)$ does not coincide with the well-known {\it Hewitt realcompactification} $\upsilon X$  induced by all the real-valued continuous functions on $X$ (see \cite{gillman}). The standard counterexamples are the closed unit ball of an infinite-dimensional separable Banach space and the metric hedgehog of countable weight $H(\omega _0)$ (\cite{engelkingbook}). Indeed, both spaces are realcompact because they are separable, that is, they coincide with their Hewitt realcompactification. On the other hand, they have the particularity that every real-valued uniformly continuous function on them is bounded and then, the Samuel realcompactification and the Samuel compactification coincide. Thus, the Samuel and the Hewitt realcompactifications are different in theses cases because otherwise  both examples would be compact, which is clearly false.
\medskip

Let us denote by $C( H(U_\mu(X)))$ the ring of all the real-valued continuous functions $f\in C(X)$ that can be continuously extended to the Samuel realcompactification $H(U_{\mu}(X))$ (\cite{hager-johnson}). The main objective of this paper is to better understand  this ring. For example, we can describe it as the family of all the real-valued continuous functions that map Cauchy filters of $(X,wU_{\mu}(X))$ to Cauchy filters of $(\Rset, d_u)$, where $d_u$ is the usual Euclidean metric on $\Rset$ (\cite{borsik}). But this kind of description does not tell us anything. 

More precisely, the question that we have is the following. We know, trivially, that all the real-valued uniformly continuous functions, as well as finite products of them, can be continuously extended to $H(U_\mu (X)))$. So, in this line, we ask which are the uniform spaces $(X,\mu)$ that satisfy that for every non-vanishing function $f\in U_\mu(X)$, that is, $f(x)\neq 0$ for every $x\in X$, the inverse function $g=1/f$ can also be continuously extended to $H(U_\mu (X))$. Observe that what we are really asking is  to know which  uniform spaces satisfy that their Samuel realcompactification $H(U_\mu (X))$ conincide with the $G_\delta$-{\it closure} of $X$ in its Samuel compactification $s_\mu X$ (see \cite{hager-johnson} and \cite{chekeev}). The $G_{\delta}$-closure of a uniform space in its Samuel compactification is also a realcompactification of the space which, a priori, does not coincide with the Samuel realcompactification or the Hewitt realcompactification (see \cite{curzer} and \cite{chekeev}). 

\medskip

In order to give an answer to the above problem we are going to study the particular case of the {\it Baire space} $D^{\kappa}$, $\kappa \geq \omega _0$. The Baire space space $D^{\kappa}$ is defined as the product of $\kappa$-many copies of a uniformly discrete space $D$. It is endowed with the product uniformity $\pi$  having as a base the uniform partitions $\{\{x\}\times D^{\kappa \backslash N}: x\in D^N\}$ where $N$ is any finite set of the ordinal set $\kappa=\{\alpha: \alpha <\kappa\}$. Observe that the Baire space $(D^\kappa, \pi)$ has the particularity that the weak uniformity $wU_{\pi}(D^{\kappa})$ needed to define the Samuel realcompactification coincides exactly with the uniformity $e\pi$ (see \cite{husek.pulga}) induced by all the uniform  partitions of  the form $\{A\times D^{\kappa \backslash N}:A\in \mathcal{A}\}$ where $\mathcal{A}$ is any countable partition of $D^N$ and $N$, as above, is a finite set of $\kappa=\{\alpha:\alpha <\kappa\}$.

By all the foregoing, our object of study, the Samuel realcompactification $H(U_{\pi}(D^{\kappa}))$ of $(D^{\kappa},\pi)$, is the topological space obtained in the completion of the uniform space $(D^{\kappa}, e\pi)$, and we want to determine if for every non-vanishing $f\in U_{\pi}(D^{\kappa})$, the inverse function $1/f$ can be continuously extended to this completion. Here, we are going to see that this problem depends  on the cardinal of  $D$, as usually results on realcompactifications do.
\medskip

First, recall that we know that the Samuel realcompactification $H(U_{\pi}(D^{\kappa}))$ coincides topologically with the original space $D^{\kappa}$ if and only if the cardinality of $D$ is {\it not  Ulam-measurable} (by \cite[Corollary 2.4]{reynolds}, or \cite[Theorem 1]{husek.pulga}). Therefore, in this case the answer is trivial and in order to have some interesting result we need to suppose that at least, the cardinal of $D$ is Ulam-measurable. In particular, we ask that the cardinal of $D$ is $\omega_1${\it-strongly compact}, even if other large-cardinal axioms could be possible, as we will explain later.

\begin{definition}  Let $\kappa\geq \omega _0$. A filter $\mathcal{F}$ satisfies the $\kappa$-{\it intersection property} if for every subfamily $\mathcal{E}\subset \mathcal{F}$ of cardinal $|\mathcal{E}|<\kappa$, then $\bigcap \mathcal{E}\neq \emptyset$. In addition, if $\bigcap \mathcal{E}\in \mathcal{F}$, we will say that $\mathcal{F}$ is $\kappa$-{\it complete}. 
\end{definition}

\noindent Clearly every filter is $\omega_0$-complete and every $\kappa$-complete filter satisfies the $\kappa$-intersection property, but not conversely. However, if an ultrafilter satisfies the $\kappa$-intersection property  then it is $\kappa$-complete.

\begin{definition} A cardinal $\kappa>\omega_0$ is {\it Ulam-measurable} if in any set of cardinal $\kappa$ contains a non-principal $\omega_1$-complete ultrafilter. It is
$\omega_1$-{\it strongly compact} if every $\kappa$-complete filter on any set $S$ can be extended to an $\omega_1$-complete ultrafilter on $S$.
\end{definition}

\noindent We will comment the implications of working with this kind of cardinals in the last section. Now, just telling that every $\omega_1$-strongly compact cardinal is Ulam-measurable. Hence we must assume their existence as a large-cardinal axiom of set-theory, as we do with Ulam-measurable cardinals, since we cannot prove it from ZFC (assuming the consistency of ZFC \cite{jech}). Moreover, if $\kappa$ is $\omega_1$-strongly measurable and $\lambda \geq \kappa$, then $\lambda$ is also $\omega _1$-strongly compact (see \cite{bagaria1}).

\medskip


Summarizing all the above, the purpose of this paper is to prove the following result.

\begin{theorem}\label{first} Let $D$ be a set of $\omega_1$-strongly compact cardinal and $\kappa \geq \omega _0$. Then there exists a non-vanishing real-valued uniformly continuous function $f$ on the Baire space $(D^{\kappa},\pi)$ such that the inverse function $g=1/f$ cannot be continuously extended to the Samuel realcompactification $H(U_\pi (D^{\kappa}))$.
\end{theorem}

Observe that, if we prove the above result, we are also proving that the porperty ``for a uniform space every inverse function of a non-vanishing real-valued uniformly function can be extended to its Samuel realcompactification'' is not productive. Indeed, a uniformly discrete space $D$ satisfies always this property, indistinctly of it is cardinal, because any continuous function on it is uniformly continuous. However, infinite products do not satisfy it whenever the cardinal of $D$ is $\omega _1$-strongly compact. Therefore, this result relates some topological/uniform object to the set-theoretic notion of $\omega_1$-strong compactness as in the line of \cite{bagaria1}, \cite{bagaria2} or \cite{usuba}. 

\end{section}
\medskip

\begin{section}{Basic facts}

In order to prove Theorem \ref{first}, instead of working with functions we are going to use a special kind of filters called {\it Cauchy $z_u$-filters}.

\begin{definition} A set $Z$ of a uniform space $(X,\mu)$ is a $z_u$-{\it set} if there exists some (bounded) real-valued uniformly continuous function $f$ such that $f^{-1}(\{0\})=Z$. 
\end{definition}

\noindent Clearly every $z_u$-set is a zero-set, but not conversely. On the other hand, in a metric space closed sets, zero-sets and $z_u$-sets are all the same. However, this is not in general true for uniform spaces.

Observe that if $Z$ is a $z_u$-set of a uniform space $(X,\mu)$ and $Y$ is a subspace of $Y$ then $Z\cap Y$ is a $z_u$-set of $(Y,\mu|_Y)$. Moreover, the sets of the form $f^{-1}([a,b])$, where $f\in U_\mu (X)$, are also examples of $z_u$-sets. Indeed consider the uniformly continuous function $h: (\Rset,d_u)\rightarrow (\Rset,d_u)$ defined by $h(x)=d_u(x,  [a,b])$. Then $h\circ f: (X,\mu)\rightarrow  (\Rset,d_u)$ is uniformly continuous and $f^{-1}([a,b])=(h\circ f)^{-1}(\{0\})$.

In the particular case of the Baire space $(D^\kappa, e\pi)$, the sets of the form   $A\times D^{\kappa\backslash N}$ where $A\subset D^N$ and $N$ is a finite set of $\kappa=\{\alpha:\alpha <\kappa\}$, are all $z_u$-sets.  Indeed let us denote by $p_N:(D^\kappa,\pi)\rightarrow (D^N,{\tt u})$ the projection of $D^\kappa$ onto the uniformly discrete space $(D^N,{\tt u})$. Then $p_N$ is a uniformly continuous map. Next, consider the uniformly continuous function $h:(D^{N},{\tt u})\rightarrow (\Rset, d_u)$ defined by $h(x)=d(x,  A)$. Then, $h\circ p_N: (D^\kappa,\pi)\rightarrow (\Rset, d_u)$ is uniformly continuous and $A\times D^{\kappa\backslash N}=(h\circ p_N)^{-1}(\{0\})$

Let us denote by $\mathcal{Z}_u(X)$ the family of all the $z_u$-filters of $(X,\mu)$.

\begin{definition} A filter $\mathcal{F}$ of a uniform space $(X,\mu)$ is a $z_u$-filter if $\mathcal{F}\cap \mathcal{Z}_u(X)$ is a base of $\mathcal{F}$. It is a $z_u$-{\it ultrafilter} if $\mathcal{F}\cap  \mathcal{Z}_u(X)$ is a maximal in $\mathcal{Z}_u(X)$.
\end{definition}

\noindent It follows from
Kuratowski-Zorn lemma that every $z_u$-filter is contained in a $z_u$-ultrafilter.

\begin{definition} A filter $\mathcal{F}$ of a uniform space $(X,\mu)$ is a {\it Cauchy filter} if for every uniform cover $\mathcal{U}\in \mu$ there is some $U\in \mathcal{U}$ such that $F\subset U$ for some $F\in \mathcal{F}$
\end{definition}

Cauchy $z_u$-filters are used in the completion of a uniform space. More precisely, if $(X,\mu)$ is a uniform space, a point in its completion $\xi$ is exactly the equivalence class induced by a {\it minimal Cauchy filter} of $(X,\mu)$ (see \cite{bourbaki}).

\begin{definition} A Cauchy filter $\mathcal{F}$ of a uniform space is {\it minimal} if does not exist a coarser Cauchy filter $\mathcal{G}\subsetneq \mathcal{F}$.

\end{definition} 

\noindent It particular, every Cauchy filter contains a unique minimal Cauchy filter (\cite{bourbaki}). Moreover, it can be shown that every minimal Cauchy filter is a Cauchy $z_u$-filter. This implies  the main fact that for every point $\xi$ in the completion of a uniform space $(X,\mu)$ there exists a Cauchy $z_u$-(ultra)filter in $(X,\mu)$ converging to $\xi$.

The next result shows us how to pass from the problem stated in terms of real-valued uniformly continuous functions to the problem with Cauchy $z_u$-filters.  

\begin{theorem}\label{Cauchy} For a uniform space $(X,\mu)$ the following statements are equivalent:
\begin{enumerate}
\item for  every non-vanishing real-valued uniformly continuous function $f\in U_\mu(X)$ the inverse function $1/f$ can be continuously extended to $H(U_{\mu}(X))$;
\smallskip

\item there is no bounded real-valeud uniformly continuous function $f\in U^{*}_\mu(X)$ such that $f(x)> 0$ for every $x\in X$ and  $F(\xi)=0$ for some $\xi \in H(U_\mu(X))$, where $F$ denotes the (unique) continuous extension of $f$ to $H(U_\mu(X))$;
\smallskip

\item every Cauchy $z_u$-filter $\mathcal{F}$ of $(X, wU_{\mu}(X)) $ satisfies the $\omega _1$-intersection property;
\smallskip

\item every minimal Cauchy filter $\mathcal{F}$ of $(X, wU_{\mu}(X)) $ satisfies the $\omega _1$-intersection property;
\smallskip

\item every Cauchy $z_u$- ultrafilter $\mathcal{F}$ of $(X, wU_{\mu}(X)) $ satisfies the $\omega _1$-intersection property.

\end{enumerate}

\begin{proof} $(1)\Rightarrow (2)$ If there is some  function $f\in U^{*}_{\mu}(X)$ such that $f(x)> 0$ for every $x\in X$ but satisfying that $F(\xi)=0$ for some $\xi \in H(U_{\mu}(X))$, where $F$ denotes the unique (\cite{garrido1}) continuous extension of $f$  to $H(U_\mu(X))$, then the continuous function $g=1/f$ cannot be continuously extended to $\xi$ as it is not defined in the point, by uniqueness of the extensions.

\smallskip

$(2)\Rightarrow (1)$ Let $f\in U_\mu (X)$ be non-vanishing. Since $U_\mu (X)$ is a vector lattice, without lose of generality we may suppose that $f> 0$. Next, let us write $f= g\cdot h$ where $g(x)=min\{ f(x),1\}$ and $h=max\{f(x),1\}$. Then the  continuous function $1/f$ can be continuously extended to $\xi \in H(U_\mu (X))$ if both continuous functions $1/h$ and $1/g$ can be continuously extended too (by uniqueness of the extensions). But this is easily seen since  $1/g$ can be continuously extended as it is uniformly continuous, and, by hypothesis, $1/h$ can be also continuously to $1/H$, where $H$ is the unique continuous extension to of the function $h$ to $H(U_\mu (X)).$

\smallskip 

$(2)\Rightarrow (3)$ Let $ \mathcal{F}$ be a Cauchy $z_u$-filter of $(X, wU_{\mu}(X)) $ and suppose, on the contrary, that for some subfamily $\{F_n:n\in \Nset\}\subset \mathcal{F}$, $$\bigcap_{n\in \Nset} F_n = \emptyset.$$ As $\mathcal{F}$ is a $z_u$-filter we may suppose that $F_{n+1}\subset F_n$ for every $n\in \Nset$ and that each $F_n$ is a $z_u$-set, that is, for every $n\in \Nset$ there exists some $f_n\in U_{\mu}(X)$ such that $F_n=f_n^{-1}(\{0\})$. Then, it is easy to check that the function $$f(x)=\sum _{n=1}^{\infty} 2^{-n}\cdot min\{|f_n(x)|,1\}$$ is  uniformly continuous and bounded. In particular, $f(x)\neq 0$ for every $x\in X$ as $\bigcap_{n\in \Nset} F_n =\emptyset$.

Let $F$ be the continuous extension of $f$ to $H(U_d(X))$. We are going to prove that there is some $\xi \in H(U_d(X))$, such that $F(\xi)=0$, contradicting like this  statement $(2)$.

Indeed, recall that $H(U_\mu(X))$ coincides topologically with the completion of $(X, wU_\mu(X))$. Therefore, if $\mathcal{F}$ is a Cauchy $z_u$-filter of $(X, wU_\mu(X))$, then $\mathcal{F}$ converges to some $\xi \in H(U_\mu(X))$. Now, by continuity,
 $$F(\xi)\in \bigcap _{n\in \Nset}F\big({\rm cl}_{H(U_d(X))}F_n\big)\subset \bigcap _{n\in \Nset} {\rm cl}_{\Rset} f(F_n)\subset\bigcap _{n\in \Nset} {\rm cl}_{\mathbb{R}}(0,1/2^n]\subset \bigcap _{n\in \Nset} [0,1/2^n]=\{0\} $$ as we wanted to show.  

\smallskip

$(3)\Rightarrow (4)$ Since every minimal Cauchy filter is a $z_u$-filter the implication follows.

\smallskip

$(4)\Rightarrow (5)$ Every Cauchy $z_u$-ultrafilter $\mathcal{F}$ contains a minimal Cauchy filter $\mathcal{G}$. If $\mathcal{G}$ satisfies the $\omega_1$-intersection property, then by \cite[Corollary 1.3]{husek},  $\mathcal{F}$ satisfies the $\omega_1$-intersection property.

$(5)\Rightarrow (2)$ Suppose, by contradiction, that there is some non-vanishing $f\in U^{*}_\mu(X)$ such that $F(\xi)=0$ for some $\xi \in H(U_\mu(X))$, where $F$ is the continuous extension of $f$ to $H(U_\mu(X))$. Since $H(U_\mu(X))$ can be described as the completion of $(X,wU_\mu(X))$, then there exists some Cauchy $z_u$-ultrafilter $\mathcal{F}$ in $(X,wU_d(X))$ converging to $\xi$. 

Then, the $z_u$-sets $f^{-1}\big((0,1/n]\big)=F^{-1}\big([0,1/n]\big)\cap X$, belongs to the $z_u$-ultrafilter $\mathcal{F}$ for every $n\in \Nset$ since, by continuity of $F$, $F^{-1}\big([0,1/n]\big)\cap F\neq \emptyset$ for every $F\in \mathcal{F}$. Thus, since $\bigcap_{n\in \Nset}f^{-1}\big((0,1/n]\big)=\emptyset$, we contradict statement $(5)$.

\end{proof}
\end{theorem}

\begin{remark} In the above theorem, statement $(2)$ is telling us that $H(U_\mu(X))$ is the $G_{\delta}$-closure of $X$ is it is Samuel compactification $s_\mu X$. Indeed, recall that $H(U_\mu(X))$ is the $G_{\delta}$-closure of $X$ in $s_\mu X$ means that every zero-set of $s_\mu X$ that meets $H(U_\mu (X))$ also meets $X$. Since the zero-sets of of $s_\mu X$ are exactly the extensions of the $z_u$-sets of $(X,\mu)$ to $s_\mu X$, the equivalence follows (see \cite{garrido1}).

\end{remark}

Coming back to the particular case that we want to study, that is, the Baire space $(D^\kappa, \pi)$, $\kappa \geq \omega_0$, recall, from the introduction, that the weak uniformity $wU_\pi (D^\kappa)$ coincides with the uniformity $e\pi$ induced by the uniform  partitions of  the form $\{A\times D^{\kappa \backslash N}:A\in \mathcal{A}\}$ where $\mathcal{A}$ is a countable partition of $D^N$ and $N$ is a finite set of $\kappa=\{\alpha:\alpha <\kappa\}$.

Moreover, it is very useful to know the following facts:

\begin{enumerate}

\item Let $(D,{\tt u})$ be a uniformly discrete space and $e{\tt u}$ the uniformity on $D$ having as a base the countable partitions of $D$.  Then a ($z_u$-)filter $\mathcal{F}$ is  Cauchy in $(D,e{\tt u})$ if and only if it is an $\omega_1$-complete ultrafilter of $D$.
\smallskip

\item A filter $\mathcal{F}$ on $(D^\kappa, e\pi)$ is Cauchy if and only if every of its projections $p_N(\mathcal{F})$ on $D^N$, where $N$ is a finite set of $\kappa=\{\alpha:\alpha<\kappa\}$, is an $\omega_1$-complete ultrafilter of $D^N$. Indeed, recall that the projection maps are uniformly continuous and then they preserve Cauchy filters (\cite{borsik}). In particular, the premigages on $D^\kappa$ of all the projections $p_N(\mathcal{F})$ is a Cauchy $z_u$-filter of $(D^{\kappa},e\pi)$ contained in $\mathcal{F}$.
\end{enumerate}

\medskip

Finally, by Theorem \ref{first} and Theorem \ref{Cauchy}, the result that we wish to prove is the following.

\begin{theorem} \label{Domega} Let $D$ be a set of $\omega_1$-strongly compact cardinal and $\kappa \geq \omega _0$. Then, there exists a Cauchy $z_u$-filter of $(D^{\kappa},e\pi)$ failing the $\omega _1$-countable intersection property.
\end{theorem}

\begin{remark} If we prove the above result, then we prove equivalently that the Samuel realcompactification $H(U_\pi(D^{\kappa}))$ do not coincide with the $G_\delta$-closure of $D^\kappa$ in its Samuel compactification $s_\pi D^\kappa$ whenever $|D|$ is $\omega_1$-strongly compact. Thus, both realcompactifications are not homeomorphic (or equivalent, see \cite{engelking.realcompact}) in this case. In particular, it follows that $H(U_\pi (D^\kappa))$ is not homeomorphic either to the Hewitt realcompactification $\upsilon D^\kappa$. However, we cannot tell if $\upsilon D^\kappa$ is  homeomorphic or not to the $G_\delta$-closure of $D^\kappa$ in its Samuel compactification. While this is always true for $\kappa=\omega _0$ and for $\kappa >\omega _0$, whenever $|D|$ is not Ulam-measurable, we don't know what happens if $k>\omega _0$ and $|D|$ is  Ulam-measurable.
\end{remark}

\end{section}

\medskip

\begin{section}{The proof}

In order to prove Theorem \ref{Domega}, we first prove it for $\kappa =\omega _0$ and next we deduce the general case from it.

\begin{theorem} \label{Domegametric} Let $D$ be a set of $\omega_1$-strongly compact cardinal. Then there exists a Cauchy $z_u$-filter of $(D^{\omega _0},e\pi)$ failing the $\omega_1$-intersection property.
\end{theorem}

The idea of the proof is the following. Let $p_n:D^{\omega _0}\rightarrow D^n$, $n\in \Nset$, be the projections onto the first $n$-coordinates, $$p_n (\langle x_1,x_2,...,x_k,...\rangle)=\langle x_1,x_2,...,x_n\rangle.$$ First, we are going to define a decreasing family of $z_u$-sets $\{F_n:n\in \Nset\}$ of $D^{\omega _0}$ such that $\bigcap _{n\in \Nset}F_n =\emptyset$ (see Theorem \ref{matrix}). To imagine this family is not very difficult, but we wish that it belongs to some Cauchy $z_u$-filter $\mathcal{F}$ of $(D^{\omega_0},e\pi)$. 

To find the filter $\mathcal{F}$ we are going to define carefully the sets $F_n$ in such a way that, for every $n\in \Nset$, the family of projections $\{p_n(F_k):k\in \Nset\}$ belongs to some $|D|$-complete filter $\mathcal{B}_n$ of $D^{n}$ satisfying that $\bigcap \mathcal{B}_n=\emptyset$. Moreover the filters $\mathcal{B}_n$, $n\in \Nset$ will be related between them as follows: $p_{n}(p_{n+1}^{-1}(\mathcal{B}_{n+1}))\subset \mathcal{B}_n$ for every $n\in \Nset$.

If this is the case, observe that, since each set $D^{n}$ has $\omega_1$-strongly compact cardinal $|D|$, the $|D|$-complete filter $\mathcal{B}_n$ can be extended to a $\omega _1$-complete ultrafilter $\mathcal{U}_n$ of $D^n$. Therefore, as we have said in the previous section, the preimages $\mathcal{F}=\bigcup\{p_{n}^{-1}(\mathcal{U}_n):n\in \Nset\}$ form a Cauchy $z_u$-filter of $(D^{\omega _0}, e\pi)$ which in particular fails the $\omega_1$-intersection property.

In this proof, the difficult task will be to define the family of $z_u$-sets $\{F_n:n\in \Nset\}$, but once we have it, we can easily prove  Theorem \ref{Domegametric} as we have just seen.

\medskip

In order to approach the above plan of proof we take into the account the following fact. Suppose that we have a family of decreasing $z_u$-sets $\{F_n:n\in \Nset\}$ such that $\bigcap _{n\in \Nset }F_n =\emptyset$. Next, let us write $B^0 _n=\bigcap_{k\geq n} p_n(F_k)$, $n\in \Nset$. If for some $n\in \Nset$, $B^0 _n=\emptyset$ then the family of projections $\{p_n (F_k):k\in \Nset\}$ does not belong to a $|D|$-complete filter of $D^n$. So we need that $B_n ^0\neq \emptyset$ for every $n\in \Nset$.

Moreover, we need to assure also that for every $n\in \Nset$, there exists some $k\geq n$ such that $p_{k}(p^{-1}_{k+1}(B^{0} _{
k+1}))\subsetneq  B^0_{k}$ because otherwise $\bigcap_{n\in \Nset} F_{n}\neq \emptyset$. Indeed, suppose that for some $n\in \Nset$ and for every $k\geq n$, $p_{k}(p^{-1}_{k+1}(B^{0} _{
k+1}))=  B^0_{k}$ and let us pick some $\langle x_1,x_2,...,x_n\rangle\in B^0_n$. Then, by definition of $B^0_n$, for every $j\in \Nset$ there exists some $z^{n,j}\in F_j$ such that $p_n(z^{n ,j})=\langle x_1,x_2,...,x_n\rangle$. Next, since $B^0 _n=p_{n}(p^{-1}_{n+1}(B^{0} _{
n+1}))$ we have that for the fixed $\langle x_1,x_2,...,x_n\rangle \in B^0 _n$ above, we can take some $x_{n+1}\in D$, such that $\langle x_1,x_2,...,x_n,x_{n+1}\rangle \in B^{0} _{
n+1}$ and such that for every $j\in \Nset$ there exists some $z^{n+1,j}\in F_j$ satisfying that $p_{n+1}(z^{n+1 ,j})=\langle x_1,x_2,...,x_n, x_{n+1}\rangle$.

If we continue this way, by induction, we arrive to a point    $\langle x_1,x_2,...,x_n,...\rangle \in \bigcap _{n\in \Nset}F_n$. Indeed, the diagonal sequence $(z^{n+j, j})_{j\in \Nset}$ of points in $D^{\omega _0}$ obtained in the induction process converges to $\langle x_1,x_2,...,x_n,...\rangle$ because $p_{n+j}(z^{n+j, j})= \langle x_1,x_2,...,x_{n+j}\rangle$, for every $j\in \Nset$,  and in addition, satisfies that $z^{n+j, j}$ belongs to the $z_u$-set $F_k$ for every $j\geq k$ and every $k\in \Nset$. Thus, we get a contradiction as $\bigcap _{n\in \Nset}F_n =\emptyset$. 

\smallskip

Summarizing all the above, we need to assure that for every $n\in \Nset$, $B^0 _n \neq\emptyset$ and that, for every $n\in \Nset$, there exists some $k\geq n$ such that $p_{k}(p^{-1}_{k+1}(B^{0} _{
k+1}))\subsetneq  B^0_{k}$.  If both conditions are satisfied we can continue and define the sets $B^{1}_n=\bigcap_{k\geq n} p_n(p^{-1} _k (B^{0}_k))$, for every $n\in \Nset$. By the same reasons as before, we ask that for every $n\in \Nset$, $B^{1}_n\neq \emptyset$ and that, for every $n\in \Nset$, there exists  some $k\geq n$ such that $ p_{k}(p_{k+1}^{-1}(B^{1} _{
k+1})) \subsetneq B^1_{k}$. 

Now, by transfinite induction, for every ordinal $\alpha<|D|$, we can define the sets $$(\diamond)\text{ } \text{ } B^{\alpha}_n=\bigcap_{\beta <\alpha}\bigcap _{k\geq n}p_n (p_k ^{-1}(B_k^{\beta}))$$  always asking that for every $n\in \Nset$, $B^{\alpha} _n \neq \emptyset$ and that\begin{center} $(\clubsuit)$ for every $n\in \Nset$, there exists some $k\in \Nset$ such that  $p_{k}(p_{k+1}^{-1}(B^{\alpha}_{k+1}))\subsetneq B_{k} ^{\alpha}$.
\end{center} 



Proceeding in this way, for every $n\in \Nset$, we have a  family of sets $\{B^{\alpha}_n: \alpha < |D|\}$ which is a filter-base for a $|D|$-complete filter $\mathcal{B}_n$ of $D^n$ satisfying in addition that the family of projections $\{p_n(F_k):k\in \Nset\}$ belongs to it, as we wished. 

Moreover we are going to ask the additional condition that for some $n\in \Nset$, $$\bigcap _{\alpha <|D|}B^{\alpha}_n =\emptyset.$$ This condition is only a requirement in order to stop the possibility of getting indefinitely the sets $B^{\alpha }_n$, $\alpha \geq |D|$, $n\in \Nset$, defined as in $(\diamond)$. Indeed, this could bring us the situation that $\bigcap _{n\in \Nset}F_n\neq \emptyset$ (see for instance \cite{husek}), and we don't want it.   
\medskip






The existence of such $z_u$-sets $F_n$, $n\in \Nset$, of $D^{\omega _0}$ satisfying  all the above requirements $(\clubsuit)$ in its reiterated projections, is proved in the next Theorem \ref{matrix}. In it we put that for every $\alpha <|D|$, $B^{\alpha}_1=D\backslash \{x_\beta:\beta <\alpha\}$. Clearly there are other possibilities, as we will comment later, but, if this is  satisfied then  we have that the property $(\clubsuit)$ is granted because we don't get stuck in the process of generating the sets $B^{\alpha} _n$. Moreover, the family $\mathcal{B}_1=\{B^{\alpha}_1:\alpha <|D|\}$ is a $\kappa$-complete filter of $D$ such that $\bigcap _{\alpha <\kappa} B^{\alpha}_1=\emptyset$, as we required.  
\medskip





Before stating Theorem \ref{matrix}, we introduce the following notation. For every $n\in \Nset$, let $A_n\subset D^n$ be non-empty subsets and let us write $\mathcal{A}=\{A_n:n\in \Nset\}$ and $$B^{0}_n(\mathcal{A})=\bigcap _{k\geq n} p_n(p_k^{-1}(A_k)) \text{ for every }n\in \Nset.$$ Then, by recursion, for every $\alpha <|D|$, where $|D|\geq \omega _0$, and every $n\in \Nset$ we define $$B^{\alpha}_n(\mathcal{A})=\bigcap _{\beta <\alpha}\bigcap _{k\geq n} p_n(p_{k}^{-1}(B_k^{\beta}(\mathcal{A}))) .$$  Observe that in particular $$B^{\alpha +1}_n(\mathcal{A})=\bigcap _{k\geq n} p_n(p_{k}^{-1}(B_k^{\alpha}(\mathcal{A}))) $$ and, whenever $\alpha$ is a limit ordinal, $$ B^{\alpha }_n(\mathcal{A})=\bigcap _{\beta<\alpha} B_n^{\beta}(\mathcal{A}) .$$



\begin{theorem} \label{matrix} Let $D$ be an infinite set and well-order it, that is, put $D=\{x_{\alpha}:\alpha<|D|\}$. Then, there exists a decreasing countable family of $z_u$-sets of $D^{\omega _0}$, of the form $F_n=p_n^{-1}(A_n)$, where $A_n\subset D^n$, for every $n\in \Nset$, and where the family of sets  $\mathcal{A}=\{A_n:n\in \Nset\}$   satisfies that:
\begin{enumerate}



\item $B_1^{\alpha}(\mathcal{A})=D\backslash\{x_{\beta}:\beta<\alpha\}$ for every $\alpha <|D|$;
\smallskip

\item $|B_n^{\alpha}(\mathcal{A})|=|D|$ for every $n\in \Nset$ and $\alpha < |D|$;
\smallskip 

\item $\bigcap _{\alpha <|D|} B^{\alpha}_n(\mathcal{A})=\emptyset$ for every $n\in \Nset$.
\end{enumerate} In particular $\bigcap_{n\in \Nset} F_n=\emptyset$.
\end{theorem}

 We proof Theorem \ref{matrix} by transfinite induction. Therefore, we need a couple of technical lemmas, one for succesor ordinals and another one for limit ordinals.

\smallskip

\begin{definition} Let $z\in D$,  $\alpha <|D|$ an ordinal and  $\mathcal{J}$ a family of sets of $D$. Suppose that a countable family of sets $\mathcal{A}(z,\mathcal{J}^{\alpha})=\{A_n(z,\mathcal{J}^{\alpha}):n\in \Nset\}$, depending on $z$, $\mathcal{J}$ and $\alpha$, has been defined. Then $\mathcal{A}(z,\mathcal{J}^{\alpha})$ satisfies the \huge{\Bat}\Large-{\it property} if :

\begin{enumerate} 

\item $A_1(z,\mathcal{J}^{\alpha})=\{z\}$
\smallskip

\item $A_n(z,\mathcal{J}^{\alpha})\subset D^n$ for every $n\in \Nset$;
\smallskip

\item $p_n(p_{n+1}^{-1}(A_{n+1}(z,\mathcal{J}^{\alpha}))) \subset A_n(z,\mathcal{J}^{\alpha})$ for every $n\in \Nset$;
\smallskip

\item $B_1^{\beta}(\mathcal{A}(z,\mathcal{J}^{\alpha}))=\{z\} $ for every $\beta \leq \alpha$;
\smallskip

\item $|B_n^{\beta}(\mathcal{A}(z,\mathcal{J}^{\alpha}))|=|D|$ for every $n\geq 2$ and every $\beta < \alpha$;
\smallskip

\item  $B_n^{\alpha }(\mathcal{A}(z,\mathcal{J}^{\alpha}))= \emptyset $ for every $n\in \Nset$, $n\geq 2$.
\end{enumerate}
\end{definition}

\begin{remark} We ask condition (5) in the above definition in order to assure condition (2) in Theorem \ref{matrix}. This is motivated by the fact that in the proof of Theorem \ref{Domegametric} we need that for every $n\in \Nset$, the ultrafilter $\mathcal{U}_n$ satisfying  the $|D|$-intersection property on $D^n$ that extends the filter-base $\{A_n\}\cup \{B_n^{\alpha}(\mathcal{A}): \alpha <|D|\}$ is free, and hence that $|U|$ is Ulam-measurable for every $U\in \mathcal{U}_n$ (see \cite{jech}).
\end{remark}

\begin{lemma} {\rm {\bf Successor ordinal.}} Let  $\alpha <|D|$ be an ordinal and $\mathcal{J}^{\alpha}$ a family of pairwise disjoint subsets of $D$. Suppose that for every $x\in D$, the family of sets $\mathcal{A}(x, \mathcal{J}^{\alpha})=\{A_n(x, \mathcal{J}^{\alpha}):n\in \Nset\}$ has been defined satisfying the  \huge{\Bat}\Large-property.  Let $\mathcal{J}=\{J_n:n\in \Nset\}$ be a family of pairwise disjoint subsets of $D$ such that $|J_n|=|D|$ for every $n\in \Nset$. Put $\mathcal{J}^{\alpha +1}=\mathcal{J}$, take $z\in D$ and define:
\begin{align*}
A_1(z, \mathcal{J}^{\alpha+1})=&\{z \}\\
A_2(z, \mathcal{J}^{\alpha+1})=&\{z\}  \times \bigcup_{k\geq 2} J_k\\
A_3(z, \mathcal{J}^{\alpha+1})=&\{z\} \times \big(\big(\bigcup_{k\geq 3} J_k\times D\big)\cup \bigcup\{A_2(x,\mathcal{J}^\alpha):x\in J_2\}\big)\\
A_4(z, \mathcal{J}^{\alpha+1})=&\{z\}\times \big(\big(\bigcup_{k\geq 4} J_k\times D^2\big)\cup \big(J_3\times \bigcup\{A_2(x,\mathcal{J}^\alpha):x\in D\}\big)\\
&\cup \bigcup\{A_3(x,\mathcal{J}^\alpha):x\in J_2\}\big)\\
\vdots \\
A_n(z, \mathcal{J}^{\alpha+1})=&\{z\} \times \big(\big(\bigcup_{k\geq n} J_k\times D^{n-2}\big)\\
&\cup \big(J_{n-1}\times D^{n-4}\times 
\bigcup\{A_2(x,\mathcal{J}^\alpha):x\in D\}\big) \\
&\cup\big(J_{n-2}\times D^{n-5}\times  \bigcup\{A_3(x,\mathcal{J}^\alpha):x\in D\}\big) \cup\ldots \\
&\cup \big(J_3\times \bigcup\{A_{n-2}(x,\mathcal{J}^\alpha):x\in D\} \big)\\
&\cup \bigcup\{A_{n-1}(x,\mathcal{J}^\alpha):x\in J_2\}\big)
\end{align*}
\begin{align*}
A_{n+1}(z, \mathcal{J}^{\alpha+1})=&\{z\} \times \big(\big(\bigcup_{k\geq n+1} J_k\times D^{n-1}\big)\\
&\cup \big(J_{n}\times D^{n-3}\times\bigcup\{A_2(x,\mathcal{J}^\alpha):x\in D\}\big)\\
&\cup \big(J_{n-1}\times D^{n-4}\times  \bigcup\{A_3(x,\mathcal{J}^\alpha):x\in D\}\big) \cup\ldots \\
&\cup \big(J_3\times \bigcup\{A_{n-1}(x,\mathcal{J}^\alpha):x\in D\} \big)\\
&\cup \bigcup\{A_{n}(x,\mathcal{J}^\alpha):x\in J_2\}\big)\\
&\vdots
\end{align*}

Then, the family of sets $\mathcal{A}(z, \mathcal{J}^{\alpha +1})=\{A_n(z, \mathcal{J}^{\alpha +1}):n\in \Nset\}$ satisfies the \huge{\Bat}\Large-property.

\begin{proof}

That $\mathcal{A}(z, \mathcal{J}^{\alpha +1})$ satisfies properties $(1), (2)$ and $(3)$ of the \huge{\Bat}\Large-property is clear from the definition of it. Therefore, we just prove $(4),$ $(5)$ and $(6)$. 

The following are easy to check: 
\begin{align*}
B_1^0(\mathcal{A}(z, \mathcal{J}^{\alpha +1}))=&\{z\}\\
B_2 ^0(\mathcal{A}(z, \mathcal{J}^{\alpha +1}))=&\{z\}  \times \bigcup_{k\geq 2} J_k\\
B_3 ^0(\mathcal{A}(z, \mathcal{J}^{\alpha +1}))=&\bigcap _{j\geq 3}\{z\} \times \big(\big(\bigcup_{k\geq 3} J_k\times D\big)\\
&\cup \bigcup\{p_2\big(p_{j-1}^{-1}(A_{j-1}(x,\mathcal{J}^{\alpha}))\big):x\in J_2\}\big)\\
=&\{z\} \times \big(\big(\bigcup_{k\geq 3} J_k\times D\big)\\
&\cup \bigcup\{ B_2 ^0(\mathcal{A}(x,\mathcal{J}^{\alpha})):x\in J_2\}\big)\\
B_4 ^0(\mathcal{A}(z, \mathcal{J}^{\alpha +1}))=&\bigcap _{j\geq 4}\{z\}\times \big(\big(\bigcup_{k\geq 4} J_k\times D^2\big)\\
&\cup \big(J_3 \times \bigcup\{p_2\big(p_{j-2}^{-1}(A_{j-2}(x,\mathcal{J}^{\alpha }))\big):x\in D\}\big)\\
&\cup\bigcup\{ p_3\big(p_{j-1}^{-1}(A_{j-1}(x,\mathcal{J}^{\alpha }))\big):x\in J_2 \}\big)\\
= &\{z\}\times \big(\big(\bigcup_{k\geq 4} J_k\times D^2\big)\\
&\cup \big(J_3\times \bigcup\{ B_2 ^0 (\mathcal{A}(x,\mathcal{J}^{\alpha })):x\in D\}\big)\\
&\cup \bigcup\{ B^{0}_3(\mathcal{A}(x,\mathcal{J}^{\alpha})):x\in J_2\}\big)\\
&\vdots\\
B_n ^0(\mathcal{A}(z, \mathcal{J}^{\alpha +1})\big)=&\bigcap _{j\geq n}\{z\} \times \big(\big(\bigcup_{k\geq n} J_k\times D^{n-2}\big)\\
&\cup\big(J_{n-1}\times D^{n-4}\times\bigcup\{p_2\big(p_{j-n+2}^{-1}(A_{j-n+2}(x,\mathcal{J}^{\alpha}))\big):x\in D\} \big)\\
&\cup\big(J_{n-2}\times D^{n-5}\times \bigcup\{p_3\big(p_{j-n+3}^{-1}(A_{j-n+3}(x,\mathcal{J}^{\alpha}))\big):x\in D\}\big) \cup\ldots 
\end{align*}
\begin{align*}
&\cup \big(J_3\times\bigcup\{p_{n-2}\big(p_{j-2}^{-1}(A_{j-2}(x,\mathcal{J}^{\alpha}))\big):x\in D\}\big)\\
&\cup \bigcup\{p_{n-1}\big(p_{j-1}^{-1}(A_{j-1}(x,\mathcal{J}^{\alpha}))\big):x\in J_2\}\big)\\
=&\{z\} \times \big(\big(\bigcup_{k\geq n} J_k\times D^{n-2}\big)\\
&\cup\big(J_{n-1}\times D^{n-4}\times \bigcup\{B_2 ^0(\mathcal{A}(x,\mathcal{J}^{\alpha})):x\in D\}\big)\\
&\cup\big(J_{n-2}\times D^{n-5}\times \bigcup\{B_3 ^0(\mathcal{A}(x,\mathcal{J}^{\alpha})):x\in D\}\big) \cup\ldots \\
&\cup \big(J_3\times \bigcup\{B_{n-2} ^0(\mathcal{A}(x,\mathcal{J}^{\alpha})):x\in D\} \big)\\
&\cup \bigcup\{B_{n-1} ^0(\mathcal{A}(x,\mathcal{J}^{\alpha})):x\in J_2\}\big)\\
B_{n+1} ^0(\mathcal{A}(z, \mathcal{J}^{\alpha +1})\big)=&\bigcap _{j\geq n+1}\{z\} \times \big(\big(\bigcup_{k\geq n+1} J_k\times D^{n-1}\big)\\
&\cup\big(J_{n}\times D^{n-3}\times\bigcup\{p_2\big(p_{j-n+1}^{-1}(A_{j-n+1}(x,\mathcal{J}^{\alpha}))\big):x\in D\} \big)\\
&\cup\big(J_{n-1}\times D^{n-4}\times \bigcup\{p_3\big(p_{j-n+2}^{-1}(A_{j-n+2}(x,\mathcal{J}^{\alpha}))\big):x\in D\}\big) \cup\ldots \\
&\cup \big(J_3\times \bigcup\{p_{n-1}\big(p_{j-2}^{-1}(A_{j-2}(x,\mathcal{J}^{\alpha}))\big):x\in D\}\big)\\
&\cup \bigcup\{p_{n}\big(p_{j-1}^{-1}(A_{j-1}(x,\mathcal{J}^{\alpha}))\big):x\in J_2\}\big)\\
=&\{z\} \times \big(\big(\bigcup_{k\geq n+1} J_k\times D^{n-1}\big)\\
&\cup\big(J_{n}\times D^{n-3}\times \bigcup\{B_2 ^0(\mathcal{A}(x,\mathcal{J}^{\alpha})):x\in D\}\big)\\
&\cup\big(J_{n-1}\times D^{n-4}\times \bigcup\{B_3 ^0(\mathcal{A}(x,\mathcal{J}^{\alpha})):x\in D\}\big) \cup\ldots \\
&\cup \big(J_3\times \bigcup\{B_{n-1} ^0(\mathcal{A}(x,\mathcal{J}^{\alpha})):x\in D\} \big)\\
&\cup \bigcup\{B_{n} ^0(\mathcal{A}(x,\mathcal{J}^{\alpha})):x\in J_2\}\big)\\
&\vdots
\end{align*}

\noindent Thus, by induction, we get that for every $\beta <\alpha$:

\begin{align*}
B_1^\beta(\mathcal{A}(z, \mathcal{J}^{\alpha +1}))=&\{z\}\\
B_2 ^\beta(\mathcal{A}(z, \mathcal{J}^{\alpha +1}))=&\{z\}  \times \bigcup_{k\geq 2} J_k\\
B_3 ^\beta(\mathcal{A}(z, \mathcal{J}^{\alpha +1}))=&\{z\} \times\big(\big(\bigcup_{k\geq 3} J_k\times D\big)\\
&\cup \bigcup\{ B_2 ^\beta(\mathcal{A}(x,\mathcal{J}^{\alpha})):x\in J_2\}\big)\\
B_4 ^\beta(\mathcal{A}(z, \mathcal{J}^{\alpha +1}))=& \{z\}\times \big(\big(\bigcup_{k\geq 4} J_k\times D^2\big)\\
&\cup \big(J_3\times \bigcup\{ B_2 ^\beta (\mathcal{A}(x,\mathcal{J}^{\alpha})):x\in D\}\big)
\end{align*}
\begin{align*}
&\cup \bigcup\{ B^{\beta}_3(\mathcal{A}(x,\mathcal{J}^{\alpha})):x\in J_2\}\big)\\
&\vdots\\
B_n ^\beta(\mathcal{A}(z, \mathcal{J}^{\alpha +1})\big)=&\{z\} \times \big(\big(\bigcup_{k\geq n} J_k\times D^{n-2}\big)\\
&\cup\big(J_{n-1}\times D^{n-4}\times \bigcup\{B_2 ^\beta(\mathcal{A}(x,\mathcal{J}^{\alpha})):x\in D\}\big)\\
&\cup\big(J_{n-2}\times D^{n-5}\times \bigcup\{B_3 ^\beta(\mathcal{A}(x,\mathcal{J}^{\alpha})):x\in D\}\big)\cup\ldots \\
&\cup \big(J_3\times \bigcup\{B_{n-2} ^\beta(\mathcal{A}(x,\mathcal{J}^{\alpha})):x\in D\} \big)\\
&\cup \bigcup\{B_{n-1} ^\beta(\mathcal{A}(x,\mathcal{J}^{\alpha})):x\in J_2\}\big)
\end{align*}
\begin{align*}
B_{n+1} ^\beta(\mathcal{A}(z, \mathcal{J}^{\alpha +1})\big)=&\{z\} \times \big(\big(\bigcup_{k\geq n+1} J_k\times D^{n-1}\big)\\
&\cup\big(J_{n}\times D^{n-3}\times \bigcup\{B_2 ^\beta(\mathcal{A}(x,\mathcal{J}^{\alpha})):x\in D\}\big)\\
&\cup\big(J_{n-1}\times D^{n-4}\times \bigcup\{B_3 ^\beta(\mathcal{A}(x,\mathcal{J}^{\alpha})):x\in D\}\big) \cup\ldots \\
&\cup \big(J_3\times \bigcup\{B_{n-1} ^\beta(\mathcal{A}(x,\mathcal{J}^{\alpha})):x\in D\} \big)\\
&\cup \bigcup\{B_{n} ^\beta(\mathcal{A}(x,\mathcal{J}^{\alpha})):x\in J_2\}\big)\\
&\vdots
\end{align*}

\noindent For $\beta =\alpha$, by hypothesis $B_{n}^{\alpha}(\mathcal{A}(x, \mathcal{J}^{\alpha +1}))=\emptyset$ for every $n\geq 2$, hence:

\begin{align*}
B_1^\alpha (\mathcal{A}(z, \mathcal{J}^{\alpha +1}))=&\{z\}\\
B_2 ^\alpha(\mathcal{A}(z, \mathcal{J}^{\alpha +1}))=&\{z\}  \times \bigcup_{k\geq 2} J_k\\
B_3 ^\alpha(\mathcal{A}(z, \mathcal{J}^{\alpha +1}))=&\{z\} \times \bigcup_{k\geq 3} J_k\times D\\
B_4 ^\alpha(\mathcal{A}(z, \mathcal{J}^{\alpha +1}))=& \{z\}\times \bigcup_{k\geq 4} J_k\times D^2\\
&\vdots\\
B_n ^\alpha(\mathcal{A}(z, \mathcal{J}^{\alpha +1})\big)=&\{z\} \times \bigcup_{k\geq n} J_k\times D^{n-2}\\
B_{n+1} ^\alpha(\mathcal{A}(z, \mathcal{J}^{\alpha +1})\big)=&\{z\} \times \bigcup_{k\geq {n+1}} J_k\times D^{n-1}\\
&\vdots
\end{align*}

It is immediate, by the characteristics of the family of sets $\mathcal{J}$, that 

$$B^{\alpha+1}_{1}(\mathcal{A}(z,\mathcal{J}^{\alpha +1}))=\{z\}\text{ and }$$ $$B^{\alpha+1}_{n}(\mathcal{A}(z,\mathcal{J}^{\alpha +1}))=\emptyset \text{ for every }n\geq 2.$$ 

\noindent Hence we conclude that  $(4)$ $(5)$ and $(6)$ are satisfied by $\mathcal{A}(z,\mathcal{J}^{\alpha +1})$ and we have finished.
\end{proof}

\end{lemma}

\begin{lemma}{\rm {\bf Limit ordinal.}} Let $\alpha<|D|$ be a limit ordinal and $\mathcal{J^\alpha}=\{J^{\beta}_n:n\in \Nset, \beta <\alpha\}$ a family of pairwise disjoint subsets of $D$ such that $|J^{\beta}_n|=|D|$. For every $\beta <\alpha$ put $\mathcal{I}^{\beta}=\{J_n ^{\beta}:n\in \Nset\}$ and suppose that the family of sets $\mathcal{A}(z,\mathcal{I}^{\beta})=\{A_n(z,\mathcal{I}^{\beta}):n\in \Nset\}$, $z\in D$, is defined satisfying the \huge{\Bat}\Large-property. Let $\mathcal{A}(z,\mathcal{J} ^{\alpha})=\{A_n (z,\mathcal{J} ^{\alpha}):n\in \Nset\}$ where

$$A_1 (z,\mathcal{J}^{\alpha})=\{z\} \text{ and}$$ $$A_n (z,\mathcal{J}^{\alpha})=\bigcup _{\beta <\alpha}A_n (z,\mathcal{I} ^{\beta}) \text{ for every }n\geq 2.$$

\noindent Then $\mathcal{A}(z,\mathcal{J} ^{\alpha})$ satisfies the \huge{\Bat}\Large -property.

\begin{proof} As in Lemma 1 we just prove that  $\mathcal{A}(z,\mathcal{J} ^{\alpha})$ satisfies properties $(4)$, $(5)$ and $(6)$ of  the \huge{\Bat}\Large-property.

Fixed $\beta \leq\alpha$, since $J^{\gamma} _{n}\cap J^{\lambda} _{k}=\emptyset$ if $\gamma \neq \lambda$, then 

$$B_{n}^{\beta}(\mathcal{A}(z,\mathcal{J} ^{\alpha}))=\bigcap _{\lambda <\beta}\bigcap _{k\geq n} p_{n}(p_{k}^{-1}(B_{k}^{\lambda}(\mathcal{A}(z,\mathcal{J} ^{\alpha}))))=$$ $$ \bigcap _{\lambda <\beta}\bigcap _{k\geq n}\bigcup_{\gamma<\alpha}p_{n}(p_{k}^{-1}(B_{k}^{\lambda}(\mathcal{A}(z,\mathcal{I}^{\gamma}))))=\bigcup_{\beta<\gamma <\alpha} B_{n}^{\beta}(\mathcal{A}(z,\mathcal{I}^{\gamma})) .$$ 

\noindent Hence $B_{1}^{\beta}(\mathcal{A}(z,\mathcal{J}^\alpha))=\{z\}$ and  

$$B_{n}^{\beta}(\mathcal{A}(z,\mathcal{J}^{\alpha}))=\bigcup_{\beta<\gamma <\alpha} B_{n}^{\beta}(\mathcal{A}(z,\mathcal{I}^{\gamma}))\text{ for every }n\geq 2.$$ Therefore, $B_{n}^{\beta}(\mathcal{A}(z,\mathcal{J}^{\alpha}))\neq \emptyset$ for every $\beta <\alpha$ and every $n\in \Nset$ because $\alpha$ is a limit ordinal. Moreover,  $B_{n}^{\alpha}(\mathcal{A}(z,\mathcal{J}^{\alpha}))= \emptyset$. Then $(4)$, $(5)$ and $(6)$ are satisfied.

\end{proof}
\end{lemma}

\begin{remark} Observe that in particular, the above lemma can also be   applied whenever $\alpha $ is a limit ordinal such that for some $\beta <\alpha$, $\beta$ is a limit ordinal too. In this case, since $|\bigcup _{n\in \Nset}J^\beta _n|=|D|$ we can arrange the family of subsets $\mathcal{I}^{\beta}=\{J^\beta _n:n\in \Nset\}$ by doing partitions on the sets $J^\beta _n$, in such a way that, after the partitions, $\mathcal{I}^{\beta}=\{{J'} ^\gamma _n:n\in \Nset, \gamma <\beta\}$  and $|{J'} ^\gamma _n|=|D|$ for every $n\in \Nset$ and every $\gamma <\beta$. 
\end{remark}

Next we prove Theorem \ref{matrix}.

\begin{proof}[Proof of Theorem \ref{matrix}] 

We are going to define the $z_u$-sets $F_n$, $n\in \Nset$.

Let $\mathcal{J}=\{J_{n}:n\in \Nset\}$ any family of pairwise disjoint subsets of $D$ such that $|J_{n}|= |D|$ for every $n\in \Nset$. Next, for every $\alpha <|D|$ limit ordinal let $\mathcal{I}^{\alpha} =\{I^\beta _n:n\in \Nset, \beta <\alpha\}$ be any family of pairwise disjoint subsets of $D$ such that $|I^{\beta}_n|=|D|$ for every $n\in \Nset$ and every $\beta <\alpha$.

We start by $\alpha =0$. Put $\mathcal{J}^0=\mathcal{J}$ and for every $x\in D$ define  $\mathcal{A}(x, \mathcal{J}^0)=\{A_n(x, \mathcal{J}^0):n\in \Nset\}$  where  
$$A_1(x, \mathcal{J}^0)=\{x\} \text{ and}$$ $$A_n(x,\mathcal{J}^0)=\{x\}\times \bigcup _{k\geq n} J^0 _n \times D^{n-2} \text{ for every } n\geq2.$$ By the characteristics of $\mathcal{J}^0$,   $\mathcal{A}(x, \mathcal{J}^0)$ satisfies the \huge{\Bat}\Large-property. 

Next, fix $\alpha <|D|$ and suppose that for every $x\in X$, every $\beta <\alpha$, and every $\mathcal{J}^\beta$ satisfying the hypothesis of Lemma 1 or Lemma 2, the family of sets $\mathcal{A}(x, \mathcal{J}^{\beta})=\{A_n(x, \mathcal{J}^\beta):n\in \Nset\}$ has been defined in such a way that $\mathcal{A}(x, \mathcal{J}^{\beta})$ satisfies the \huge{\Bat}\Large-property. If $\alpha$ is a successor ordinal for some $\beta$, that is $\alpha=\beta +1$, then we let $\mathcal{J}^{\alpha}=\mathcal{J}$ and we define $\mathcal{A}(x, \mathcal{J}^\alpha)$ applying Lemma 1. Otherwise, if $\alpha$ is a limit ordinal, we put $\mathcal{J}^\alpha=\mathcal{I}^\alpha$ and we define $\mathcal{A}(x, \mathcal{J}^\alpha)$ applying Lemma 2.
\smallskip

Now, let  $\mathcal{P}=\{P_n:n\in \Nset\}$ be a partition of $D$ such that $|P_n|=|D|$. Then if $\alpha$ is a successor ordinal, we let $\mathcal{P}^{\alpha}=\mathcal{P}$. Otherwise, we choose a partition $\mathcal{P}^{\alpha}=\{P^\beta _n:n\in \Nset, \beta <\alpha\}$ of $D$ such that $|\mathcal{P}^{\beta}_n|=|D|$ for every $n\in \Nset$ and every $\beta <\alpha$.

Next,  well-order $D$, that is, $D= \{x_{\alpha}:\alpha <|D|\}$. Then, for every $\alpha <|D|$ we take the families of sets $\mathcal{A}(x_{\alpha},\mathcal{P}^{\alpha})=\{A_n(x_{\alpha},\mathcal{P}^{\alpha}):n\in \Nset\}$ and for every $n\in \Nset$ we put $$A_n=\bigcup_{\alpha<|D|} A_n(x_{\alpha},\mathcal{P}^{\alpha})$$ and $$F_n=p_{n}^{-1}(A_n).$$ Clearly $\{F_n:n\in \Nset\}$ is a dcreasing family of $z_u$-sets.

Moreover, let $\mathcal{A}=\{A_n:n\in \Nset\}$. Since each family of sets $\mathcal{A}(x_{\alpha}, \mathcal{P}^{\alpha})$ satisfies the  \huge{\Bat}\Large-property then, it is clear that  $$B^{\alpha}_1(\mathcal{A})= \bigcup_{\beta \geq\alpha, \beta <|D|} B^{\alpha}_1(\mathcal{A}(x_{\beta},\mathcal{P}^{\beta}))=D\backslash\{x_{\beta}:\beta <\alpha\},$$
$$B^{\alpha}_n(\mathcal{A})= \bigcup_{\beta >\alpha, \beta <|D|} B^{\alpha}_n(\mathcal{A}(x_{\beta},\mathcal{P}^{\beta})) \text{ and then}$$ $$ \bigcap _{\alpha <|D|}B_{n}^{\alpha}(\mathcal{A})=\emptyset.$$ Hence, conditions $(1)$, $(2)$ and $(3)$ are also satisfied.
\smallskip

Finally, assume that $\bigcap_{n\in \Nset} F_n\neq\emptyset$, that is, there exists some point $$\langle z_1,z_2,...,z_n,...\rangle\in \bigcap_{n\in \Nset} F_n.$$ Then, it is easy to see that $\langle z_1,...,z_n\rangle\in B^\alpha_{n}(\mathcal{A})$ for every $\alpha <|D|$ and every $n\in \Nset$. Thus, $$\langle z_1,....,z_n\rangle\in \bigcap_{\alpha <|D|}B^\alpha_{n}(\mathcal{A}).=\emptyset$$ which is a contradiction. Therefore, $\bigcap_{n\in \Nset} F_n=\emptyset$.

\end{proof}

Now we are ready to prove Theorem \ref{Domegametric} and Theorem \ref{Domega}.

\begin{proof}[Proof of Theorem \ref{Domegametric}]

By Theorem \ref{matrix} there exists a countable family of $z_u$-sets  $\{F_n:n \in \Nset\}$ of $D^{\omega_0} $ such that $\bigcap_{n\in \Nset} F_n=\emptyset$ and such that, for every $n\in \Nset$, the family of projections $\{p_n(F_k):k\in \Nset\}$ belongs to the filter-base $\{B_n ^{\alpha}(\mathcal{A}):\alpha <|D|\}$ inducing a $|D|$-complete filter $\mathcal{B}_n$ of $D^{n}$. Moreover we have that for every $n\in \Nset$, $p_{n}(p_{n+1}^{-1}(\mathcal{B}_{n+1}))\subset \mathcal{B}_n$ for every $n\in \Nset$.

Next, since each set $D^{n}$ has $\omega_1$-strongly compact cardinal $|D|$, the $|D|$-complete filter $\mathcal{B}_n$ can be extended to a $\omega _1$-complete ultrafilter $\mathcal{U}_n$ of $D^n$. Then the preimages $\mathcal{F}=\bigcup\{p_{n}^{-1}(\mathcal{U}_n):n\in \Nset\}$ form a filter of $D^{\omega _0}$ because $p_{n}(p_{n+1}^{-1}(\mathcal{B}_{n+1}))\subset \mathcal{B}_n$ for every $n\in \Nset$. Moreover, as we have said previously, $\mathcal{F}$ is a Cauchy $z_u$-filter of $(D^{\omega _0}, e\pi)$ which in particular fails the $\omega_1$-intersection property as $F_n\in \mathcal{F}$ for every $n\in \Nset$.
\end{proof}

\begin{proof}[Proof of Theorem \ref{Domega}] 

Let $\kappa \geq\omega _0$ and $D$ a set of $\omega_1$-strongly compact cardinal.

First, observe that the Baire space $(D^{\kappa},\pi)$ contains as a uniform copy of $(D^{\omega _0},\pi)$, precisely, the subspace $Y=\prod_{\alpha <\kappa} Y_{\alpha}$ where $Y_\alpha =D$ for every $\alpha <\omega _0$ and $Y_{\alpha}=\{x\}$ for every $\alpha \geq \omega _0$, $\alpha <\kappa$, where $x$ is a fixed point of $D$.  Then, the inclusion map $i:(Y,\pi|_Y)\rightarrow (D^{\kappa},\pi)$ is uniformly continuous. In particular the inclusion map $i:(Y,e(\pi|_Y))\rightarrow (D^{\kappa},e\pi)$ is also uniformly continuous and hence every Cauchy filter of $(Y,e(\pi|_Y))$ is also a Cauchy filter of $(D^{\kappa},e\pi)$.
 
Next, by Theorem \ref{Domegametric} and Theorem \ref{matrix} there exists a Cauchy $z_u$-filter $\mathcal{F}$ on $(Y,e(\pi|_Y))$ containing $z_u$-sets $F_n\in \mathcal{F}$, $n\in \Nset$, of the form $F_n=p_n^{-1}(A_n)$, where $A_n \subset D^n$, which satisfy that $\bigcap F_n=\emptyset$. Then $\mathcal{F}$ is also a Cauchy filter of $(D^{\kappa},e\pi)$. Let us project the filter $\mathcal{F}$ onto each set $D^{N}$ where $N$ is a finite set of $\{\alpha:\alpha <\kappa\}$. Then the preimages of this projections are a Cauchy $z_u$-filter $\mathcal{F}'$ of $(D^{\kappa},e\pi)$. Moreover the $z_u$-sets $F'_n$  given by the primages on $D^{\kappa}$ of the sets $A_n$, $n\in \Nset$, belongs to $\mathcal{F}'$ and satisfies that $\bigcap_{n\in \Nset} F'_n=\emptyset$, that is, $\mathcal{F}'$ fails the $\omega_1$-intersection property.

\end{proof}

\begin{remark} In Theorem \ref{Domegametric} and in Theorem \ref{Domega} we have more precisely proved that there exists a point $\xi$ in the Samuel realcompactification of $(D^{\kappa},e\pi)$ which does not belong to the $G_\delta$-closure of $D^{\kappa}$ in its Samuel compactification $s_\pi D^k$. It is exactly the convergence point of the Cauchy $z_u$-ultrafilter of $(D^\kappa, e\pi)$ failing the $\omega_1$-intersection property. Moreover, we can assure there are infinitely-many points like this lying in the remainder. Indeed, we can apply Theorem \ref{matrix} to any set $E\subset D$ of cardinal $|E|=|D|$ from an infinite partition of $D$.

\end{remark}

\begin{remark} The above results Theorem \ref{Cauchy}, Theorem \ref{Domegametric} and Theorem \ref{Domega}  cannot be stated for Cauchy filters which are not $z_u$-filters. For instance, if $D$ is a set of cardinal at least two there exists a Cauchy filter on $(D^{\omega},e\pi)$ failing the $\omega _1$-intersection property. To show that take a point $x\in D^{\omega _0}$ and sets of the form $A_n=p^{-1}_n(p_n(x))$, $n\in \Nset$. Then, the sets $F_n=A_n\backslash \{x\}$ form a subbase of a filter $\mathcal{F}$ of $D^{\omega _0}$ failing the $\omega_1$-intersection property. By completeness of $(D^{\omega _0},\pi)$, $\mathcal{F}$ is a Cauchy filter of $(D^{\omega _0},\pi)$ because it converges to $x$. In particular, it is also Cauchy for $(D^{\omega _0},e\pi)$ as the uniformity $e\pi$ is weaker than $\pi$. 
\end{remark}

\end{section}

\begin{section}{Final remarks}

In this paper we have proved that the following implications are satisfied: $$\text{the cardinal of }D \text{ is }\omega _1\text{-strongly compact }$$ $$\Downarrow$$ \begin{center}
there exists Cauchy $z_u$-filter in $(D^\kappa,e\pi)$, $\kappa \geq \omega _0$ failing the $\omega _1$-intersection property 
\end{center}
$$\Downarrow$$
$$\text{the cardinal of }D \text{ is Ulam-measurable }$$

Now, we ask which of the above implications can be reversed. A first answer could be the following. Observe that in \cite{magidor} (see \cite{bagaria1} and \cite{bagaria2}), it is shown  that, assuming the consistency of ZFC together with the large-cardinal axiom `` there exists and $\omega_1$-strongly compact cardinal'', then it is also consistent with ZFC that the first $\omega_1$-strongly compact cardinal is the first Ulam-measurable cardinal. If this is so, we have that the above implications are equivalences.

However, we also have a different situation. Indeed, in \cite{bagaria1} (see also \cite{gitik}) it is proved that, assuming the consistency of ZFC with a stronger large-cardinal axiom that states that  ``there exists a {\it supercompact cardinal}'', then it is also consistent ZFC together with the fact that the first $\omega_1$-strongly compact  cardinal is strictly grater than the first  Ulam-measurable cardinal.

In this case, the requirement in Theorem \ref{Domega} that the cardinal of $D$ is $\omega _1$-strongly compact  could bee too strong. Indeed, looking into the proof of Theorem \ref{Domega} and Theorem \ref{Domegametric} it is enough to ask that the cardinal of $D$ is {\it $\omega _1$-strongly measurable}, that is, on every set of cardinal $|D|$, every $|D|$-complete filter can be extended to an $\omega _1$-complete ultrafilter. This notion of $\omega_1$-strong mesurability generalizes the concept on {\it strong-measurable cardinal}  that can be found in \cite{comfort}, as $\omega_1$-strong compatness generalizes {\it strong compactness}, and we don't know if it has ever been considered. So we ask if there exists a model, assuming the consistence of ZFC with some large cardinal axiom, in which all the above implications are not reversed. Other possibility is that it is enough to work with Ulam-measurable cardinals. However, we don't have any idea of a possible proof of this fact.
\medskip

Anyway, before ending, observe that if $D$ is a set of cardinal $\kappa _1$ satisfying that $(D^{\omega _0},e\pi)$ contains a Cauchy $z_u$-filter which fails the $\omega_1$-intersection property, then, for any set $S$ of cardinal $\kappa _2 \geq \kappa _1$, the Baire space $(S^{\omega _0},e\pi)$ contains also a  Cauchy $z_u$-filter failing the $\omega_1$-intersection property. Indeed, $(D^{\omega _0},\pi)$ is a closed uniform  subspace of $(S^{\omega _0},\pi)$. Then, any $z_u$-filter of $(D^{\omega _0},\pi)$ is also a $z_u$-filter of $(S^{\omega _0},\pi)$. Moreover, the incusion map $i:(D^{\omega _0},e\pi)\rightarrow (S^{\omega _0},e\pi)$ is uniformly continuous. Therefore, if $\mathcal{F}$ is a Cauchy $z_u$-ultrafilter of $(D^{\omega _0},e\pi)$ failing the $\omega_1$-intersection property, the same ultrafilter works for $(S^{\omega _0},e\pi)$.

\bigskip

\bigskip

\noindent {\it Acknowledgments}. I would like to thank Prof. M. Hu\v{s}ek for his useful comments and remarks.

\end{section}

\end{document}